\documentclass[12pt,a4paper]{article}

\usepackage{amsmath}
\usepackage{amssymb}
\usepackage{amsthm}
\usepackage{graphicx}
\usepackage{epsfig}
\usepackage{nccmath}

\usepackage[margin=2cm]{geometry}

\let\oldproofname=\proofname
\renewcommand{\proofname}{\rm\bf{\oldproofname}}

\newtheorem{theorem}{Theorem}
\newtheorem{lemma}[theorem]{Lemma}
\newtheorem{corollary}[theorem]{Corollary}

\newtheorem*{conjecture}{Conjecture}
\theoremstyle{definition}
\newtheorem*{defn}{Definition}
\newtheorem*{question}{Question}

\def \deg {{\rm deg}}

\def \Nbd {{\rm Nbd}}
\def \mod#1{{\:({\rm mod}\ #1)}}

\title{\bf Embedding partial Steiner triple systems with few triples}

\author{
Daniel Horsley\\
School of Mathematical Sciences \\
Monash University \\
Vic 3800, Australia \\[0.1cm]
\texttt{danhorsley@gmail.com}}

\date{ }

\begin{document}
\maketitle\thispagestyle{empty}
\def\baselinestretch{1.5}\small\normalsize

\begin{abstract}
It was proved in 2009 that any partial Steiner triple system of order $u$ has an embedding of
order $v$ for each admissible integer $v\geq 2u+1$. This result is best-possible in the sense
that, for each $u\geq 9$, there exists a partial Steiner triple system of order $u$ that does
not have an embedding of order $v$ for any $v<2u+1$. Many partial Steiner triple systems do have
embeddings of orders smaller than $2u+1$, but little has been proved about when these embeddings
exist. In this paper we construct embeddings of orders less than $2u+1$ for partial Steiner
triple systems with few triples. In particular, we show that a partial Steiner triple system of
order $u \geq 62$ with at most $\frac{u^2}{50}-\frac{11u}{100}-\frac{116}{75}$ triples has an
embedding of order $v$ for each admissible integer $v \geq \frac{8u+17}{5}$.
\end{abstract}

\section{Introduction}

A \emph{Steiner triple system of order $v$} is a pair $(V,\mathcal{B})$ where $V$ is a set with
$v$ elements and $\mathcal{B}$ is a set of (unordered) triples chosen from $V$ such that every
(unordered) pair of elements of $V$ is a subset of exactly one triple. In 1847, Kirkman \cite{Ki}
proved that there exists a Steiner triple system of order $v$ if and only if $v\equiv 1,3\mod{6}$.
Such integers are called {\it admissible}. A \emph{partial Steiner triple system of order $u$} is
a pair $(U,\mathcal{A})$ where $U$ is a set with $u$ elements and $\mathcal{A}$ is a set of
triples chosen from $U$ such that every pair of elements of $U$ is a subset of at most one triple.
The \emph{leave} of a partial Steiner triple system $(U,\mathcal{A})$ is the graph $L$ with vertex
set $U$ and edge set given by $xy\in E(L)$ if and only if the pair $\{x,y\}$ occurs in no triple
in $\mathcal{A}$. An \emph{embedding} of a partial Steiner triple system $(U,\mathcal{A})$ is a
(complete) Steiner triple system $(V,\mathcal{B})$ such that $U\subseteq V$ and
$\mathcal{A}\subseteq \mathcal{B}$. If $V=U$ the embedding is a \emph{completion}. The
\emph{embedding spectrum} of a partial Steiner triple system is the set of all orders for which it
has an embedding. A \emph{triangle decomposition} of a graph $G$ is a set of triangles whose edge
sets partition the edge set of $G$. A Steiner triple system of order $v$ can be considered as a
triangle decomposition of a complete graph of order $v$.

In 1977 Lindner \cite{LiEv} conjectured that any partial Steiner triple system of order $u$ has an
embedding of order $v$ for each $v\geq 2u+1$ such that $v \equiv 1,3 \mod{6}$. This is
best-possible in the sense that, for each $u\geq 9$, there exists a partial Steiner triple system
of order $u$ that does not have an embedding of order $v$ for any $v<2u+1$ \cite{CoRo}. Most
existing results on embedding partial Steiner triple systems concern embeddings of order at least
$2u+1$. In particular, a series of results \cite{Tr,Li,AnHiMe,Br3uEmbed} guaranteed the existence
of progressively smaller embeddings, culminating in a complete proof of Lindner's conjecture in
\cite{BrHoEmbed}.

\begin{theorem}[\cite{BrHoEmbed}]\label{LindnerTh}
Any partial Steiner triple system of order $u$ has an embedding of order $v$ for each $v\geq 2u+1$
such that $v \equiv 1, 3 \mod{6}$.
\end{theorem}

Of course, many partial Steiner triple systems do have embeddings of order less than $2u+1$. Much
less is known about when these embeddings exist and they have only been constructed for limited
classes of partial Steiner triple systems. Colbourn, Colbourn and Rosa \cite{CoCoRo} showed that
there is a completion of every partial Steiner triple system of order $u$ which has exactly $u-2$
triples and the property that every triple contains at least one of two specified elements. Bryant
\cite{Br1} gave necessary and sufficient conditions for a partial Steiner triple system of order
$u$ every vertex of whose leave has degree $d$ or $0$ to have an embedding of order $u+d$, and
found the complete embedding spectrum in the case $d=2$. Bryant, Maenhaut, Quinn and Webb
\cite{BrMaQuWe} found the embedding spectra of all partial Steiner triple systems of order $10$
with $3$-regular leaves.  The main result of \cite{BrHoSubsys} determines the embedding spectra of
partial Steiner triple systems whose leaves are complete bipartite graphs. Emphasising the
difficulty of this problem in general, Colbourn \cite{Co} has shown that it is NP-complete to
decide whether a partial Steiner triple system has an embedding of order less than $2u+1$.
Recently, numerous embeddings of order less than $2u+1$ have been constructed for Steiner triple
systems whose leaves have few edges and low maximum degree \cite{HoEmbed}.


This paper focusses on embeddings of partial Steiner triple systems with few triples. We show that
every partial Steiner triple system of order at least 62 which satisfies an upper bound on its
number of triples has an embedding of order $v$ for each admissible integer $v \geq
\frac{8u+17}{5}$.

\begin{theorem}\label{MainTheorem}
Any partial Steiner triple system of order $u \geq 62$ with at most
$\frac{u^2}{50}-\frac{11u}{100}-\frac{116}{75}$ triples has an embedding of order $v$ for each
integer $v \geq \frac{8u+17}{5}$ such that $v \equiv 1, 3 \mod{6}$.
\end{theorem}

The bound of $\frac{8u+17}{5}$ is an artefact of the techniques applied in constructing the
embeddings. The lower bound on $u$ could be reduced, but at the expense of further complicating
the proof. As a complement to this result, it follows from Corollary \ref{LBCor} that for
sufficiently large values of $u$ there is a partial Steiner triple system of order $u$ with at
most $\frac{u^2}{50}-\frac{11u}{100}-\frac{116}{75}$ triples that does not have an embedding of
order $v$ for any $v < (1.346)u$.

In fact, we will prove a result stronger than Theorem \ref{MainTheorem} concerning triangle
decompositions of graphs of order $v$ having many vertices of degree $v-1$.

\begin{theorem}\label{NWTheorem}
Let $G$ be a graph such that every vertex of $G$ has even degree and $|E(G)| \equiv 0 \mod{3}$. If
$G$ has order $v \geq 103$, $|E(G)| \geq
\binom{v}{2}-(\frac{3v^2}{128}-\frac{27v}{64}-\frac{409}{128})$ and at least $\frac{3v+17}{8}$
vertices of $G$ have degree $v-1$, then there is a triangle decomposition of $G$.
\end{theorem}

It is worth noting that this result settles the following conjecture, due to Nash-Williams
\cite{Nas}, in the case of graphs of this very specific form (it also includes some cases not
covered by the conjecture).

\begin{conjecture}[\cite{Nas}]\label{NWConj}
Let $G$ be a graph such that every vertex of $G$ has even degree and $|E(G)| \equiv 0 \mod{3}$. If
$G$ has minimum degree at least $\frac{3}{4}|V(G)|$, then there is a triangle decomposition of
$G$.
\end{conjecture}

\section{Background and Notation}\label{BackgroundAndNotation}

In this section we introduce some terminology and concepts which we will need in later sections.

All graphs considered in this paper are simple and loopless. For a nonnegative integer $v$, a
complete graph of order $v$ will be denoted by $K_v$ and, for a given set $V$, the complete graph
on vertex set $V$ will be denoted by $K_V$. For nonnegative integers $u$ and $w$, a complete
bipartite graph with parts of size $u$ and $w$ will be denoted by $K_{u,w}$ and, for disjoint sets
$U$ and $W$, the complete bipartite graph with parts $U$ and $W$ will be denoted by $K_{U,W}$. For
graphs $G$ and $H$, we denote by $G \cup H$ the graph with vertex set $V(G) \cup V(H)$ and edge
set $E(G) \cup E(H)$, we denote by $G \cap H$ the graph with vertex set $V(G) \cap V(H)$ and edge
set $E(G) \cap E(H)$, we denote by $G-H$ the graph with vertex set $V(G)$ and edge set $E(G)
\setminus E(H)$, and, if $V(G)$ and $V(H)$ are disjoint, we denote by $G \vee H$ the graph $G\cup
H \cup K_{V(G),V(H)}$ (our use of this last notation will imply that $V(G)$ and $V(H)$ are
disjoint).

For a graph $G$ and a vertex $x$ of $G$, the set of vertices adjacent in $G$ to $x$ is denoted by
$\Nbd_G(x)$, the degree in $G$ of $x$ is denoted by $\deg_G(x)$, and the maximum degree of $G$ is
denoted by $\Delta(G)$. A graph $G$ is \emph{even} if $\deg_G(x)$ is even for each $x \in V(G)$,
and is \emph{$r$-regular} if $\deg_G(x)=r$ for each $x \in V(G)$. A \emph{matching} is a
$1$-regular graph, a cycle is a connected $2$-regular graph, and a \emph{path} is a connected
graph with two vertices of degree 1 and every other vertex of degree 2 (we allow trivial matchings
with no vertices or edges but demand that cycles and paths have at least two vertices). The
\emph{length} of a cycle or path is the number of edges it has. We let $(x_1,\ldots,x_p)$ denote
the cycle with vertex set $\{x_1,\ldots,x_p\}$ and edge set
$\{x_1x_2,x_2x_3,\ldots,x_{p-1}x_p,x_px_1\}$. A cycle of length $3$ is called a \emph{triangle}.

A \emph{decomposition} of a graph $G$ is a set of subgraphs of $G$ whose edge sets partition the
edge set of $G$. If every graph in a decomposition is a cycle we call the decomposition a
\emph{cycle decomposition} and if every graph in a decomposition is a triangle we call the
decomposition a \emph{triangle decomposition}. A \emph{packing} of a graph $G$ is a decomposition
of a subgraph $H$ of $G$, and the graph $G-H$ is the \emph{leave} of the packing. Note that if $L$
is the leave of a cycle packing of a graph $G$, then $\deg_L(x) \equiv \deg_G(x) \mod{2}$ for each
$x \in V(G)$. It is clear that a Steiner triple system of order $v$ can be considered as a
triangle decomposition of $K_v$ and that a partial Steiner triple system of order $u$ can be
considered as a triangle packing of $K_u$. If $(U,\mathcal{A})$ is a partial Steiner triple system
of order $u$ and $L$ is the leave of $(U,\mathcal{A})$, then an embedding of $(U,\mathcal{A})$ of
order $u+w$ can be considered as a triangle decomposition of $L \vee K_w$.

A {\it proper edge colouring} of a graph $G$ with $c$ colours can be considered as a decomposition
of $G$ into $c$ matchings. The minimum number of matchings in any such decomposition of $G$ is
called the \emph{chromatic index} of $G$ and is denoted by $\chi'(G)$. Vizing \cite{Vi} proved
that $\chi'(G) \in \{\Delta(G),\Delta(G)+1$\} for any graph $G$.


\begin{defn}
If $L$ is a graph of order $u$ and $w$ is a nonnegative integer we will say that $(L,w)$ is
\emph{admissible} if
\begin{itemize}
    \item
$u+w$ is odd;
    \item
$\deg_L(x) \equiv u+1 \mod{2}$ for each $x \in V(L)$; and
    \item
$|E(L)|+uw+\frac{w(w-1)}{2} \equiv 0 \mod{3}$.
\end{itemize}
\end{defn}
When $L$ is the leave of a partial Steiner triple system, $(L,w)$ is admissible if and only if
$u+w \equiv 1, 3 \mod{6}$. The following lemma gives some well-known necessary conditions for the
existence of an embedding of a partial Steiner triple system.
\begin{lemma}[\cite{Br1}]\label{NecConds}
Let $L$ be a graph and let $w$ be a nonnegative integer such that there is a triangle
decomposition of $L \vee K_w$. Then $(L,w)$ is admissible and there is a subgraph $L'$ of $L$ such
that
\begin{itemize}
    \item
there is a triangle decomposition of $L-L'$;
    \item
$\chi'(L') \leq w$; and
    \item
$|E(L')| \geq \frac{w(u-w+1)}{2}$.
\end{itemize}
\end{lemma}

\begin{defn}
Let $\mathcal{P}$ and $\mathcal{P}'$ be cycle packings of a graph $G$ and let $S$ be a subset of
$V(G)$. We say that $\mathcal{P}$ and $\mathcal{P}'$ are {\em equivalent on $S$} if we can write
$\mathcal{P}=\{C_1,\ldots,C_t\}$ and $\mathcal{P}'=\{C'_1,\ldots,C'_t\}$ such that
\begin{itemize}
    \item
for $i \in \{1,\ldots,t\}$, $|E(C_i)|=|E(C'_i)|$ ;
    \item
for each $x \in S$ and for $i \in \{1,\ldots,t\}$, $x \in V(C_i)$ if and only if $x \in
V(C'_i)$; and
    \item
for each $xy \in E(G)$ such that $x,y \in S$ and for $i \in \{1,\ldots,t\}$, $xy \in E(C_i)$ if
and only if $xy \in E(C'_i)$.
\end{itemize}
\end{defn}


\section{Preliminary Results}

Most of our effort will go into proving Lemma \ref{singleEmbedding} which shows that, for any
sufficiently dense graph $L$ of order $u \geq 62$, there is a triangle decomposition of $L \vee
K_w$ if $(L,w)$ is admissible and $w = \tfrac{3u+k}{5}$ for some $k \in \{17,19,21,23\}$ (see
Section \ref{constructionSection}). Once this is established the general result follows relatively
easily (see Section \ref{mainResultSection}). The crucial point is that any odd integer $v$ can be
written as $u'+w'$ for some integers $u'$ and $w'$ such that $w' = \tfrac{3u'+k'}{5}$ for some $k'
\in \{17,19,21,23\}$.

Note that, if $u$ and $w$ are integers such that $w = \tfrac{3u+k}{5}$ for some $k \in
\{17,19,21,23\}$, then $u \not\equiv 0 \mod{5}$ and
    $$k = \left\{
    \begin{array}{ll}
        17, & \hbox{if $u \equiv 1 \mod{5}$;} \\
        19, & \hbox{if $u \equiv 2 \mod{5}$;} \\
        21, & \hbox{if $u \equiv 3 \mod{5}$;} \\
        23, & \hbox{if $u \equiv 4 \mod{5}$.}
    \end{array}
    \right.$$
It is a routine matter to verify inequalities in $u$ and $w$ under these conditions, and we will
often leave such tasks to the reader in what follows. Note in particular that, under these
conditions, $w \equiv u+1 \mod{2}$ and if $u \geq 11$ then $10 \leq w \leq u-1$.

In this section of the paper we establish a number of preliminary results that will be required
for our main construction in Section \ref{constructionSection}. We first require two results,
namely Lemmas \ref{getSubgraph} and \ref{getSubgraphEasy}, that, given a graph, find a subgraph of
it whose components are even-length paths and even-length cycles. The proofs of both of these
lemmas require Lemma \ref{disjointMatchings}, which is a well-known and easy consequence of the
max-flow min-cut theorem \cite{EliFeiSha,ForFul}.

\begin{lemma}\label{disjointMatchings}
Let $H$ be a bipartite graph with parts $A$ and $B$. Then $H$ contains two edge-disjoint matchings
$M_1$ and $M_2$ with $|E(M_1)|+|E(M_2)|=2|A|-d$ if and only if $2|S| - \sum_{x \in B}
\max(2,|\Nbd_H(x)\cap S|) \leq d$ for each subset $S$ of $A$.
\end{lemma}

\begin{proof}
Let $a$ and $b$ be vertices not in $V(H)$. Let $N$ be a flow network on node set $V(H) \cup
\{a,b\}$ with source $a$ and sink $b$, such that the underlying graph of $N$ is $K_{\{a\},A} \cup
H \cup K_{B,\{b\}}$, each edge in $K_{\{a\},A}$ has capacity $2$ and is directed from $a$ to a
vertex in $A$, each edge in $H$ has capacity $1$ and is directed from a vertex in $A$ to a vertex
in $B$, and each edge in $K_{B,\{b\}}$ has capacity $2$ and is directed from a vertex in $B$ to
$b$. There is an obvious correspondence between $k$-flows in $N$ and subgraphs of $H$ with maximum
degree 2 and size $k$. Further, any such subgraph of $H$ can be decomposed into two edge-disjoint
matchings. So it follows from the max-flow min-cut theorem that $H$ contains two edge-disjoint
matchings $M_1$ and $M_2$ such that $|E(M_1)|+|E(M_2)|=2|A|-d$ if and only if the capacity over
every $a$-$b$ cut in $N$ is at least $2|A|-d$. It is not difficult to check that this is the case
if and only if $2|S| - \sum_{x \in B} \max(2,|\Nbd_H(x)\cap S|) \leq d$ for each subset $S$ of
$A$.
\end{proof}

\begin{lemma}\label{getSubgraph}
Let $L^*$ be a graph of order $u \geq 22$, and let $w$ be an integer such that $w =
\tfrac{3u+k}{5}$ for some $k \in \{17,19,21,23\}$. If $|E(L^*)| \geq \frac{w(u-w+1)}{2}$ and
$\Delta(L^*) \leq w-8$, then $L^*$ contains a subgraph $F$ such that $|E(F)|=u-w+1$ and each
component of $F$ is an even-length path or an even-length cycle.
\end{lemma}

\begin{proof} Among all partitions of $V(L^*)$ into one part of size $\frac{u-w+1}{2}$ and one part of
size $\frac{u+w-1}{2}$, let $\{A,B\}$, where $|A|=\frac{u-w+1}{2}$ and $|B|=\frac{u+w-1}{2}$, be
one such that $|E(L^*) \cap E(K_B)|$ is minimised. Let $R$ be the largest subset of $A$ for which
$L^* \cap K_{R,B}$ contains two edge-disjoint matchings $M_1$ and $M_2$ such that
$|E(M_1)|=|E(M_2)|=|R|$. If $R=A$ it can be seen that $M_1 \cup M_2$ is a subgraph of $L^*$ with
the required properties, so we can assume that $R \neq A$. It is not difficult to see from the
definition of $R$ that the sizes of any two edge-disjoint matchings in $L^* \cap K_{A,B}$ add to
at most $|A|+|R|$. It follows by Lemma \ref{disjointMatchings} that there is a subset $S$ of $A$
such that $d \geq |A|-|R|$, where $d=2|S| - \sum_{x \in B} \max(2,|\Nbd_{L^*}(x)\cap S|)$.

Let $s=|S|$, $B'=\{x \in B: |\Nbd_{L^*}(x) \cap S| =1\}$ and $B''=\{x \in B: |\Nbd_{L^*}(x) \cap
S| \geq 2\}$. In the case $s \geq 2$, it can be seen that
$$\medop\sum_{x \in S}|\Nbd_{L^*}(x) \cap B| \leq s|B''| + |B'| \leq s(s-\tfrac{d}{2}),$$
where the second inequality follows from $2|B''|+|B'|=2s-d$. Hence, in this case, $|\Nbd_{L^*}(y)
\cap B| \leq s-\frac{d}{2}$ for some $y \in S$. In the case $s=1$, it is easy to see that
$|\Nbd_{L^*}(y) \cap B| \leq 1$ for the unique vertex $y \in S$. In either case, if there were a
vertex $z$ in $B$ such that $|\Nbd_{L^*}(z) \cap B| > |\Nbd_{L^*}(y) \cap B|$ then the partition
$\{(A \setminus \{y\}) \cup \{z\},(B \setminus \{z\}) \cup \{y\}\}$ would contradict our
definition of $\{A,B\}$. If $|\Nbd_{L^*}(y) \cap B| \leq 1$, this implies that $|E(L^*)| \leq
\Delta(L^*)|A|+\frac{1}{2}|B|=(w-8)\frac{u-w+1}{2}+\frac{u+w-1}{4}$, which can be seen to
contradict $|E(L^*)| \geq \frac{w(u-w+1)}{2}$ for $u \geq 11$. Thus we can assume that $s \geq 2$,
that $|\Nbd_{L^*}(x) \cap B| \leq s-\frac{d}{2}$ for each $x \in B$, and, since $\sum_{x \in B}
\max(2,|\Nbd_{L^*}(x)\cap S|) \geq |\Nbd_{L^*}(y) \cap B| \geq 2$, that $d \leq 2s-2$.

Let $F$ be a subgraph of $L^*$ such that $M_1 \cup M_2$ is a subgraph of $F$, each component of
$F$ is an even-length path or an even-length cycle, and subject to these conditions $F$ is of
maximum size. If $|E(F)| \geq u-w+1$ then we are finished, so we can suppose for a contradiction
that $|E(F)| \leq u-w-1$ (note that $|E(F)|$ and $u-w+1$ are even). Let $\{B_0,B_1,B_2\}$ be the
partition of $B$ such that $B_i = \{x \in B: \deg_F(x)= i\}$ for $i \in \{0,1,2\}$. Define a
partition $\{E_1,E_2,E_3,E_4,E_5\}$ of $E(L^*)$ as follows.
\begin{align*}
  E_1 &= \{xy \in E(L^*): x \in A \setminus S\} \\
  E_2 &= \{xy \in E(L^*): x \in S, y \notin A \setminus S\} \\
  E_3 &= \{xy \in E(L^*): x \in B_2, y \notin A\} \\
  E_4 &= \{xy \in E(L^*): x \in B_0, y \notin A \cup B_2\} \\
  E_5 &= \{xy \in E(L^*): x,y \in B_1\}
\end{align*}
We will find an upper bound on the size of each part of this partition and use these to obtain a
contradiction to $|E(L^*)| \geq \frac{w(u-w+1)}{2}$.

Because $\Delta(L^*)\leq w-8$, we have $|E_1| \leq (\frac{u-w+1}{2}-s)(w-8)$. We saw earlier that
$\sum_{x \in S}|\Nbd_{L^*}(x) \cap B| \leq s(s-\tfrac{d}{2})$, and thus $|E_2| \leq
\binom{s}{2}+s(s-\frac{d}{2})$. We also saw that $|\Nbd_{L^*}(x) \cap B| \leq s-\frac{d}{2}$ for
each $x \in B$, and so we have that $|E_3|\leq |B_2|(s-\frac{d}{2})$ and $|E_5| \leq
\frac{1}{2}|B_1|(s-\frac{d}{2})$. If there were a vertex $z_1$ in $B_0$ which was incident with
two edges $z_1z_2$ and $z_1z_3$ such that $z_2,z_3 \in B_0 \cup B_1$, then the graph with edge set
$E(F) \cup \{z_1z_2,z_1z_3\}$ would contradict our definition of $F$, and it follows that $|E_4|
\leq |B_0|$. Thus, because $|E_1|+|E_2|+|E_3|+|E_4|+|E_5|=|E(L^*)|=\frac{w(u-w+1)}{2}$,
    $$\theta=\tfrac{w(u-w+1)}{2} - ((\tfrac{u-w+1}{2}-s)(w-8) + (\tbinom{s}{2}+s(s-\tfrac{d}{2}))
    + |B_2|(s-\tfrac{d}{2}) + |B_0| + \tfrac{1}{2}|B_1|(s-\tfrac{d}{2}))$$
is nonpositive. We will obtain a contradiction by showing that $\theta$ is positive.

Because $|B_0|+|B_1|+|B_2|=|B|=\frac{u+w-1}{2}$, we can substitute
$|B_0|=\frac{u+w-1}{2}-|B_1|-|B_2|$ and simplify to obtain
$$\theta=\tfrac{1}{2}(s(2w+d-3s-15) - (2|B_2|+|B_1|)(s-\tfrac{d}{2}-1)+|B_1|+7u-9w+9).$$
Note that $2|R|+2|B_2|+|B_1|=2|E(F)| \leq 2(u-w-1)$ and thus, because $d \geq
|A|-|R|=\frac{u-w+1}{2}-|R|$, we have $2|B_2|+|B_1| \leq u-w-3+2d$. Using this fact, along with
$|B_1| \geq 0$ and simplifying we see
$$\theta \geq \tfrac{1}{4}d(u-w-2s+2d+1)+\tfrac{1}{2}s(3w-u-3s-12)+4u-5w+3,$$
(note that $s-\frac{d}{2}-1 \geq 0$ because $d \leq 2s-2$). Since $s \leq \frac{u-w+1}{2}$ and $d
\geq 1$, we have that $\tfrac{1}{4}d(u-w-2s+2d+1) \geq \tfrac{1}{4}(u-w-2s+3)$, and consequently
we have
$$\theta \geq \tfrac{1}{2}s(3w-u-3s-13)+\tfrac{1}{4}(17u-21w+15).$$
Since $2 \leq s \leq \frac{u-w+1}{2}$ and $u \geq 22$, it can be seen that $6 \leq 3s \leq
3w-u-19$, and consequently that $\tfrac{1}{2}s(3w-u-3s-13) \geq 3w-u-19$. Thus,
$$\theta \geq \tfrac{1}{4}(13u-9w-61)$$
and so, since $u \geq 14$, $\theta$ is positive as required. \end{proof}

\begin{lemma}\label{getSubgraphEasy}
Let $L^*$ be a graph of order $u \geq 16$, and let $w$ be an integer such that $w =
\tfrac{3u+k}{5}$ for some $k \in \{17,19,21,23\}$. If $|E(L^*)| \geq \frac{w(u-w+1)}{2}$,
$\Delta(L^*) \leq w-8$, and $A$ is a set of $\frac{u-w+1}{2}$ vertices of $L^*$ such that
$\deg_{L^*}(x) \geq u-w$ for each $x \in A$, then $L^*$ contains a subgraph $F$ such that
$|E(F)|=u-w+1$, $A \subseteq V(F)$, and each component of $F$ is an even-length path or an
even-length cycle.
\end{lemma}

\begin{proof} Let $B = V(L^*) \setminus A$ and consider the graph $L^* \cap K_{A,B}$. Since
$\deg_{L^*}(x) \geq u-w$ for each $x \in A$, we have that $|\Nbd_{L^*}(x) \cap B| \geq u-w-(|A|-1)
\geq \frac{u-w+1}{2}$ for each $x \in A$. Clearly then, $\sum_{x \in B} \max(2,|\Nbd_{L^*}(x)\cap
S|) \geq u-w+1 \geq 2|S|$ for each $S \subseteq A$ with $|S| \geq 2$, and $\sum_{x \in B}
\max(2,|\Nbd_{L^*}(x)\cap S|) \geq \frac{u-w+1}{2} \geq 2|S|$ for each $S \subseteq A$ with $|S| =
1$ (note that $u \geq 16$ implies $\frac{u-w+1}{2} \geq 2$). Thus, by Lemma
\ref{disjointMatchings}, there are two edge-disjoint matchings $M_1$ and $M_2$ in $L^*$ such that
$|E(M_1)|+|E(M_2)|=2|A|=u-w+1$. It can be seen that $M_1 \cup M_2$ is a subgraph of $L^*$ with the
required properties.
\end{proof}

We will also require Lemma \ref{getTriangle} which is similar to Lemma 5.1 of \cite{BrHoEmbed},
except that it concerns general cycle decompositions of graphs rather than simply partial Steiner
triple systems. Lemma \ref{getNose} is used only in the proof of Lemma \ref{getTriangle}. These
two lemmas are proved using path switching techniques as encapsulated below in Lemma
\ref{PathSwitchGeneral} which appears as Lemma 9 in  \cite{HoEmbed}.

\begin{lemma}[\cite{HoEmbed}]\label{PathSwitchGeneral}
Let $G$ be a graph, let $\mathcal{P}$ be a cycle packing of $G$, let $L$ be the leave of
$\mathcal{P}$, and let $a$ and $b$ be vertices in $G$ such that $\Nbd_G(a)=\Nbd_G(b)$. Then there
exists a partition of the set $(\Nbd_L(a)\cup \Nbd_L(b)) \setminus ((\Nbd_L(a)\cap \Nbd_L(b)) \cup
\{a,b\})$ into pairs such that for each pair $\{x,y\}$ of the partition, there exists a cycle
packing $\mathcal{P}_{\{x,y\}}$ of $G$ which is equivalent to $\mathcal{P}$ on $V(G) \setminus
\{a,b\}$ and whose leave $L_{\{x,y\}}$ differs from $L$ only in that $a x$, $a y$, $b x$ and $b y$
are edges in $L_{\{x,y\}}$ if and only if they are not edges in $L$.
\end{lemma}

\begin{lemma}\label{getNose}
Let $L$ be a graph, let $T$ be a set of vertices which is disjoint from $V(L)$, let $\mathcal{P}$
be a cycle packing of $L \vee K_T$, let $G$ be the leave of $\mathcal{P}$, and let $a,b,c \in T$
and $d \in V(L)$ be distinct vertices such that $ab,cd \in E(G)$. Let $\mathcal{B} = \{C \in
\mathcal{P}: \mbox{$C=(b,y,z)$ for some $y \in V(L)$, $z \in T$}\}$ and suppose that each edge in
$E(K_T) \setminus E(G)$ which is incident with $b$ is in a triangle in $\mathcal{B}$ and that
$|\Nbd_G(x) \cap T| \geq 1$ for each vertex $x$ in $V(L)$ which is in a triangle in $\mathcal{B}$.
Then there is a cycle packing of $L \vee K_T$ which is equivalent to $\mathcal{P}$ on $V(L) \cup
\{b\}$ and whose leave contains the edges $bc$ and $cd$.
\end{lemma}

\begin{proof} If $bc \in E(G)$ then we are finished immediately, so we may assume that $bc \notin E(G)$.
Let $y_0=d$ and $z_1=c$. The hypotheses of the lemma guarantee that we can create a sequence
$y_0,z_1,y_1,\ldots,z_t,y_t,z_{t+1}$ of vertices in $V(L) \cup T$ such that
\begin{itemize}
    \item
$y_iz_{i+1} \in E(G)$ for $i \in \{0,\ldots,t\}$;
    \item
$(b,z_i,y_i) \in \mathcal{P}$ for $i \in \{1,\ldots,t\}$; and
    \item
$z_1,\ldots,z_t,b$ are distinct vertices in $T$, $y_0,\ldots,y_t$ are distinct vertices in $U$,
and either $bz_{t+1} \in E(G)$ or $z_{t+1} \in \{z_1,\ldots,z_t\}$.
\end{itemize}

{\bf Case 1.} Suppose that $bz_{t+1} \in E(G)$ (this includes the case where $z_{t+1}=a$). Then
let
    $$\mathcal{P}'=(\mathcal{P} \setminus \{(b,y_i,z_i):i\in\{1,\ldots,t\}\}) \cup \{(b,y_i,z_{i+1}):i\in\{1,\ldots,t\}\}.$$
It is routine to check that $\mathcal{P}'$ is a cycle packing of $L \vee K_T$ with the required
conditions.

{\bf Case 2.} Suppose that $z_{t+1}=z_s$ for some $s \in \{1,\ldots,t\}$. If the $(a,z_s)$-switch
in $\mathcal{P} \setminus \{(b,y_s,z_s)\}$ with origin $y_{s-1}$ does not have terminus $y_s$,
then let $\sigma$ be this switch. Otherwise, let $\sigma$ be the $(a,z_s)$-switch in $\mathcal{P}
\setminus \{(b,y_s,z_s)\}$ with origin $y_t$. It follows from its definition that $\sigma$ does
not terminate at $y_{s-1}$. Note also that $\sigma$ does not terminate at $b$ because $b$ is
adjacent to both $a$ and $z_s$ in the leave of $\mathcal{P} \setminus \{(b,y_s,z_s)\}$. Unless
$s=1$ and $\sigma$ has origin $y_0$, let $\mathcal{P}^{\dag}$ be the cycle packing of $L \vee K_T$
obtained by applying $\sigma$ to $\mathcal{P} \setminus \{(b,y_s,z_s)\}$ (we will deal with the
exceptional case shortly). Consider the cycle packing $\mathcal{P}^{\dag} \cup \{(b,y_s,z_s)\}$ of
$L \vee K_T$. We are now in a situation similar to Case 1, with the relevant vertex sequence being
$y_0,z_1,y_1,\ldots,z_{s-1},y_{s-1},a$ if $\sigma$ has origin $y_{s-1}$ and being
$y_0,z_1,y_1,\ldots,z_t,y_t,a$ otherwise. Thus, we can complete the proof as in Case 1 (note that
$\mathcal{P}^{\dag} \cup \{(b,y_s,z_s)\}$ is equivalent to $\mathcal{P}$ on $V(L) \cup (T
\setminus \{a,z_s\})$ by Lemma \ref{PathSwitchGeneral}). In the case where $s=1$ and $\sigma$ has
origin $y_0$, let $\mathcal{P}^{\dag}$ be the cycle packing of $L \vee K_T$ obtained by applying
the $(a,z_1)$-switch with origin $y_1$ to $\mathcal{P} \setminus \{(b,y_1,z_1)\}$ (this switch
does not terminate at $y_0$, because it is not $\sigma$) and observe that $\mathcal{P}^{\dag} \cup
\{(a,b,y_1)\}$ has the required properties (note that $\mathcal{P}^{\dag} \cup \{(a,b,y_1)\}$ is
equivalent to $\mathcal{P}$ on $V(L) \cup (T \setminus \{a,z_1\})$ by Lemma
\ref{PathSwitchGeneral}). \end{proof}

\begin{lemma}\label{getTriangle}
Let $L$ be a graph, let $T$ be a set of vertices which is disjoint from $V(L)$, let $\mathcal{P}$
be a cycle packing of $L \vee K_T$, and let $G$ be the leave of $\mathcal{P}$. Suppose that
\begin{itemize}
    \item
$|E(G) \cap E(K_T)| \geq 1$;
    \item
every cycle in $\mathcal{P}$ that contains an edge in $E(K_T)$ is a triangle with one vertex in
$V(L)$; and
    \item
$|\Nbd_G(x) \cap T| \geq 1$ for each $x \in V(L)$, and $|\Nbd_G(y) \cap T| \geq 2$ for some $y
\in V(L)$.
\end{itemize}
Then there is a cycle packing $\mathcal{P}' \cup \{C\}$ of $L \vee K_T$ such that $C$ is a
triangle with $V(C) \cap V(L)=\{y\}$ and $\mathcal{P}'$ is equivalent to $\mathcal{P}$ on $V(L)$.
\end{lemma}

\begin{proof} {\bf Case 1.} Suppose there are vertices $q$ and $r$ in $T$ such that $ry, qr \in
E(G)$. Let $s$ be a vertex in $(\Nbd_G(y) \cap T) \setminus \{r\}$. If $s=q$ then $\mathcal{P}
\cup \{(q,r,y)\}$ is a cycle packing of $L \vee K_T$ with the required properties, so we may
assume that $s \neq q$. Apply Lemma \ref{getNose} to $\mathcal{P}$, taking $a=q$, $b=r$, $c=s$ and
$d=y$, to obtain a cycle packing $\mathcal{P}'$ of $L \vee K_T$ which is equivalent to
$\mathcal{P}$ on $V(L) \cup \{r\}$ and whose leave contains the edges $rs$ and $sy$. Note that the
leave of $\mathcal{P}'$ also contains the edge $ry$ because $\mathcal{P}'$ is equivalent to
$\mathcal{P}'$ on $V(L) \cup \{r\}$. Then $\mathcal{P}' \cup \{(r,s,y)\}$ is a cycle packing of $L
\vee K_T$ with the required properties.

{\bf Case 2.} Suppose that no vertex in $\Nbd_G(y) \cap T$ is adjacent in $G$ to another vertex in
$T$. Let $r$ be a vertex in $\Nbd_G(y) \cap T$ and let $pq$ be an edge in $E(G) \cap E(T)$. Apply
Lemma \ref{getNose} to $\mathcal{P}$ with $a=p$, $b=q$, $c=r$ and $d=y$, to obtain a cycle packing
$\mathcal{P}'$ of $L \vee K_T$ which is equivalent to $\mathcal{P}$ on $V(L) \cup \{q\}$ and whose
leave contains the edges $qr$ and $ry$. We are now in a situation covered by Case 1 and can
complete the proof as we did there. \end{proof}

Finally we will require Lemma 18 from \cite{HoEmbed}. The phrasing has been altered substantially
so as to avoid the need to introduce technical definitions from that paper. This lemma allows us
to obtain a triangle decomposition of a graph $L \vee K_w$ from a cycle packing of $L \vee
K_{w-5}$ possessing very specific properties.

\begin{lemma}[\cite{HoEmbed}]\label{finishOff}
Let $L$ be a graph of order $u$, let $w$ be an integer such that $\frac{3u+17}{5} \leq w \leq
u-1$, and let $T$ be a set of $w-5$ vertices which is disjoint from $V(L)$. Suppose that $(L,w)$
is admissible and there exists a cycle packing $\mathcal{P}$ of $L \vee K_T$ with a leave $F \cup
G$ such that the following hold.
\begin{itemize}
    \item
$F$ is a subgraph of $L$ such that every component of $F$ is an even-length path or an
even-length cycle, and $G$ is a subgraph of $K_{V(L),T}$.
    \item
Each cycle of length at least 4 in $\mathcal{P}$ is a subgraph of $K_{D,T}$, for some proper
subset $D$ of $V(L) \setminus V(F)$ such that each vertex in $D$ is in exactly one such cycle
and $|D| = \frac{5w-3u-13}{2}$.
    \item
For some distinct vertices $a_1, a_2 \in D$, we have
$$\deg_{G}(x)=
\left\{
  \begin{array}{ll}
    \deg_{F}(x)+1, & \hbox{if $x \in V(F)$;} \\
    3, & \hbox{if $x \in D \setminus \{a_1,a_2\}$;} \\
    1, & \hbox{if $x \in (U \setminus (V(F) \cup D)) \cup \{a_1,a_2\}$.}
  \end{array}
\right.$$
\end{itemize}
Then there exists a triangle decomposition of $L \vee K_w$.
\end{lemma}

\section{Construction}\label{constructionSection}


In this section we will prove that there is a triangle decomposition of $L \vee K_w$ for any
sufficiently dense graph $L$ of order $u \geq 62$ and any integer $w$ such that $(L,w)$ is
admissible and $w = \tfrac{3u+k}{5}$ for some $k \in \{17,19,21,23\}$. Our approach is first to
find a triangle packing of $L$ whose leave $L^*$ has exactly $\frac{w(u-w+1)}{2}$ edges and has
maximum degree at most $w-8$ (see Lemma \ref{makeLeaveSparse}). We then construct a triangle
decomposition of $L^* \vee K_w$ (see Lemma \ref{singleSparseLeaveEmbedding}). The union of these
sets of triangles is a triangle decomposition of $L \vee K_w$. Choosing
$|E(L^*)|=\frac{w(u-w+1)}{2}$ means that a triangle decomposition of $L^* \vee K_w$ that has no
triangle with all three vertices in $V(L^*)$ also has no triangle with no vertices in $V(L^*)$.
This makes possible a proof of Lemma \ref{singleSparseLeaveEmbedding} based on the repeated
application of Lemma \ref{getTriangle}.

We prove Lemma \ref{makeLeaveSparse} by first taking a maximum triangle packing of a complete
graph on vertex set $V(L)$ and deleting from this packing any triangles which are not subgraphs of
$L$. We then adjust the resulting packing to ensure that it has the required properties.

\begin{lemma}\label{makeLeaveSparse}
Let $L$ be a graph of order $u \geq 62$, let $w$ be an integer such that $w = \tfrac{3u+k}{5}$ for
some $k \in \{17,19,21,23\}$, and suppose that $(L,w)$ is admissible. If $|E(L)| \geq
\binom{u}{2}-\frac{w(u-w+1)-u-2}{4}$, then there is a triangle packing of $L$ with a leave $L^*$
such that $|E(L^*)|=\frac{w(u-w+1)}{2}$ and $\Delta(L^*) \leq w-8$.
\end{lemma}

\begin{proof} Let $U=V(L)$, let $L^c=K_U-L$, and note that $|E(L^c)| \leq \frac{w(u-w+1)-u-2}{4}$. By the
main result of \cite{Sc} there is a triangle packing $\mathcal{M}$ of $K_U$ whose leave $G$
contains at most $\frac{u+2}{2}$ edges (note that $u \geq 6$). Let
    $$\mathcal{P}^{\dag}=\{C \in \mathcal{M}:E(C) \subseteq E(L)\}$$
be a triangle packing of $L$, and let $H^{\dag}$ be the leave of $\mathcal{P}^{\dag}$. For each $C
\in \mathcal{M} \setminus \mathcal{P}$, we have that $|E(C) \cap E(L)| \leq 2$, $|E(C) \cap
E(L^c)| \geq 1$, and no other triangle in $\mathcal{M}$ contains any edge in $E(C) \cap E(L^c)$.
Thus,
    $$|E(H^{\dag})| \leq 2|E(L^c)|+|E(G)| \leq \tfrac{w(u-w+1)}{2},$$
using $|E(L^c)| \leq \frac{w(u-w+1)-u-2}{4}$ and $|E(G)| \leq \frac{u+2}{2}$.
So, because $(L,w)$ being admissible implies that $|E(L)| \equiv \frac{w(u-w+1)}{2} \mod{3}$, we
can delete triangles from $\mathcal{P}^{\dag}$ to obtain a triangle packing $\mathcal{P}$ of $L$
such that the leave $H$ of $\mathcal{P}$ has exactly $\frac{w(u-w+1)}{2}$ edges. Note that
$\deg_H(x) \equiv \deg_L(x) \equiv w \mod{2}$ for each $x \in U$, using the fact that $(L,w)$ is
admissible. Because $|\mathcal{P}| = \tfrac{1}{3}(|E(L)|-|E(H)|)$ it can be seen to follow from
our hypothesis $|E(L)| \geq \binom{u}{2}-\frac{w(u-w+1)-u-2}{4}$ that
    $$|\mathcal{P}| \geq \tfrac{1}{12}(u(2u-3w-1)+3w(w-1)+2).$$

If $\Delta(H) \leq w-8$ then $\mathcal{P}$ is a triangle packing of $L$ with the required
properties, so we may assume $\Delta(H) \geq w-6$. It suffices to show that we can find a triangle
packing $\mathcal{P}'$ of $L$ with a leave $H'$ such that $|E(H')|=\frac{w(u-w+1)}{2}$,
$\Delta(H') \leq \Delta(H)$ and $H'$ has fewer vertices of degree $\Delta(H)$ than $H$, because
then, by repeating this procedure, we will eventually obtain a triangle packing of $L$ with the
required properties.

Let $a$ be a vertex such that $\deg_{H}(a)=\Delta(H)$, let $A=\Nbd_{H}(a)$, let $S=\{x \in U:
\deg_{H}(x) \geq w-8\}$ and let $\mathcal{Q}=\{C \in \mathcal{P}: |V(C) \cap S| \geq 1\}$. Clearly
$|S| \leq \frac{w(u-w+1)}{w-8}$. Also, note that each vertex in $S$ occurs in at most
$\frac{1}{2}((u-1)-(w-8))=\frac{u-w+7}{2}$ triangles in $\mathcal{P}$. The proof now splits into
two cases according to whether $E(H) \cap E(K_A) = \emptyset$.

{\bf Case 1.} Suppose that there is an edge $xy$ in $E(H) \cap E(K_A)$. Because $|S| \leq
\frac{w(u-w+1)}{w-8}$ and each vertex in $S$ occurs in at most $\frac{u-w+7}{2}$ triangles in
$\mathcal{P}$, we have that $|\mathcal{Q}| \leq \frac{w(u-w+1)(u-w+7)}{2(w-8)}$. Thus, using our
lower bound on $|\mathcal{P}|$ it is routine to check that $|\mathcal{P}| > |\mathcal{Q}|$ for $u
\geq 52$. It follows that there is a triangle $C$ in $\mathcal{P} \setminus \mathcal{Q}$. It is
easy to see that
$$\mathcal{P}' = (\mathcal{P} \setminus \{C\}) \cup \{(x,y,a)\}$$
is a triangle packing of $L$ with the required properties.

{\bf Case 2.} Suppose that $E(H) \cap E(K_A) = \emptyset$. It follows that each vertex in $A$ has
degree in $H$ at most $u-|A| \leq u-w+6$, which for $u \geq 42$ implies degree at most $w-10$.
Thus, $A$ and $S$ are disjoint sets, and $|S| \leq u-|A| \leq u-w+6$. Because $|S \setminus \{a\}|
\leq u-w+5$ and each vertex in $S$ occurs in at most $\frac{u-w+7}{2}$ triangles in $\mathcal{P}$,
we have that at most $\frac{(u-w+5)(u-w+7)}{2}$ triangles in $\mathcal{P}$ contain a vertex in $S
\setminus \{a\}$ and hence that at most $\frac{(u-w+5)(u-w+7)}{2}$ edges of $K_A$ are in triangles
in $\mathcal{Q}$. Obviously, at most $|E(L^c)| \leq \frac{w(u-w+1)-u-2}{4}$ edges of $K_A$ are in
$E(L^c)$. Thus, because each edge of $K_A$ is either in $E(L^c)$ or in a triangle in $\mathcal{P}$
and because it is routine to check that, for $u \geq 62$,
$$|E(K_A)| \geq \mbinom{w-6}{2} > \mfrac{(u-w+5)(u-w+7)}{2}+\mfrac{w(u-w+1)-u-2}{4},$$
we have that some edge $xy \in E(K_A)$ is in a triangle $(x,y,z)$ in $\mathcal{P} \setminus
\mathcal{Q}$. It is easy to see that
$$\mathcal{P}' = (\mathcal{P} \setminus \{(x,y,z)\}) \cup \{(x,y,a)\}$$
is a triangle packing of $L$ with the required properties. \end{proof}

In order to prove Lemma \ref{singleSparseLeaveEmbedding} we first use a proper edge colouring of a
subgraph $L^*$ to obtain a cycle packing of $L^* \vee K_{w-5}$ consisting of triangles, each of
which contains exactly two vertices in $V(L^*)$, and one or two longer cycles. We then repeatedly
apply Lemma \ref{getTriangle} until we obtain a cycle packing of $L^* \vee K_{w-5}$ to which we
can apply Lemma \ref{finishOff} and thus complete the proof. Lemma \ref{singleSparseLeaveHelper}
is a technical result that is used only in the proof of Lemma \ref{singleSparseLeaveEmbedding}.

\begin{lemma}\label{singleSparseLeaveHelper}
Let $L^*$ be a graph of order $u \geq 32$ and let $w$ be an integer such that $w = \tfrac{3u+k}{5}$
for some $k \in \{17,19,21,23\}$. If $|E(L^*)|=\frac{w(u-w+1)}{2}$, $\Delta(L^*) \leq w-8$ and
$\deg_{L^*}(x) \equiv u+1 \mod{2}$ for each $x \in V(L^*)$, then there is a decomposition
$\{F,F_1,\ldots,F_{w-7}\}$ of $L^*$ such that
\begin{itemize}
    \item
$|E(F)|=u-w+1$ and each component of $F$ is an even-length path or an even-length cycle;
    \item
$F_i$ is a matching for $i \in \{1,\ldots,w-7\}$; and
    \item
there is a proper subset $D$ of $U \setminus V(F)$ such that $|D|=\frac{5w-3u-13}{2}$, $|D
\setminus V(F_{w-7})| \geq 2$, and $\deg_{L^*}(x) \leq w-10$ for each $x \in D$.
\end{itemize}
\end{lemma}

\begin{proof} Let $U=V(L^*)$ and let $M$ be a subset of $U$ such that $|M|=\frac{u-w+1}{2}$ and
$\deg_{L^*}(x) \geq \deg_{L^*}(y)$ for each $x \in M$ and $y \in U \setminus M$. By Lemma
\ref{getSubgraph}, there is a subgraph $F$ of $L^*$ such that $|E(F)|=u-w+1$ and each component of
$F$ is an even-length path or an even-length cycle. Further, by Lemma \ref{getSubgraphEasy}, we
can assume that $M \subseteq V(F)$ if $\deg_{L^*}(x) \geq u-w$ for each $x \in M$. Since
$\Delta(L^*-F) \leq \Delta(L^*) \leq w-8$ there is a decomposition $\{F_1,\ldots,F_{w-7}\}$ of
$L^*-F$ into $w-7$ matchings by Vizing's theorem \cite{Vi}.

It is easy to check that $u-|V(F)| \geq 6$ using $|V(F)| \leq \frac{3}{2}|E(F)|$. Let $D'$ be a
subset of $U \setminus V(F)$ such that $|D'|=5$ and $\deg_{L^*}(x) \leq \deg_{L^*}(y)$ for each $x
\in D'$ and $y \in U \setminus (V(F) \cup D')$. Note that $D'$ is a proper subset of $U \setminus
V(F)$. It suffices to show that $\deg_{L^*}(x) \leq w-10$ for each $x \in D'$ and that $|D'
\setminus V(F_i)| \geq 2$ for some $i \in \{1,\ldots,w-7\}$. This is because we can reindex the
matchings in $\{F_1,\ldots,F_{w-7}\}$, if necessary, so that $|D' \setminus V(F_{w-7})| \geq 2$,
and take $D$ to be a subset of $D'$ such that $|D|=\frac{5w-3u-13}{2}$ and $|D \setminus
V(F_{w-7})| \geq 2$ (note that $\frac{5w-3u-13}{2} \in \{2,3,4,5\}$). The proof splits into two
cases according to whether $\deg_{L^*}(x) \geq u-w$ for each $x \in M$.

{\bf Case 1.} Suppose that $\deg_{L^*}(x) \leq u-w-1$ for some $x \in M$. It is easy to check that
$u-|V(F)| \geq |M| + |D'|$ using $|V(F)| \leq \frac{3}{2}|E(F)|$, and this implies that $M \cap
D'=\emptyset$ and hence that $\deg_{L^*}(x) \leq u-w-1$ for each $x \in D'$. Since $u-w-1 \leq
w-10$ for $u \geq 11$, we have $\deg_{L^*}(x) \leq w-10$ for each $x \in D'$. Furthermore, using
the fact that $\deg_{L^*}(x) \leq u-w-1$ for each $x \in D'$, it can be shown that
$$\medop\sum_{x \in D'} |\{i:i \in \{1,\ldots,w-7\},x \notin V(F_i)\}| \geq 5(w-7)-5(u-w-1) > w-7,$$
which implies that $|D' \setminus V(F_i)| \geq 2$ for some $i \in \{1,\ldots,w-7\}$. (Informally,
we are evaluating the sum over the vertices in $D'$ of the number of matchings ``missing'' at a
vertex and then concluding that some matching must be missing at two vertices in $D'$.)

{\bf Case 2.} Suppose that $\deg_{L^*}(x) \geq u-w$ for each $x \in M$. Then $M \subseteq V(F)$.
Consider the set $S=(U \setminus V(F)) \cup M$ and note that $|S| \geq w-1$ using $|V(F)| \leq
\frac{3}{2}|E(F)|$. The definitions of $D'$ and $M$ imply that $\deg_{L^*}(x) \leq \deg_{L^*}(y)$
for each $x \in D'$ and $y \in S \setminus D'$. Also note that $\sum_{x \in S} \deg_{L^*}(x) \leq
(w-1)(u-w+1)$, since $|E(L^*)| = \frac{w(u-w+1)}{2}$ and $\sum_{x \in V(F) \setminus M} \deg_F(x)
\geq 2|E(F)|-2|M| \geq u-w+1$. So, because $|S| \geq w-1$ and $|D'| = 5$, it can be seen that
$\max(\{\deg_{L^*}(x):x\in D'\}) \leq \frac{(w-1)(u-w+1)}{w-5}$ and $\sum_{x \in D'} \deg_{L^*}(x)
\leq 5(u-w+1)$. As in Case 1, it is routine to use these facts to show that, for $u \geq 32$,
$\deg_{L^*}(x) \leq w-10$ for each $x \in D$ and that $|D' \setminus V(F_i)| \geq 2$ for some $i
\in \{1,\ldots,w-7\}$. \end{proof}

\begin{lemma}\label{singleSparseLeaveEmbedding}
Let $L^*$ be a graph of order $u \geq 32$ and let $w$ be an integer such that $w = \tfrac{3u+k}{5}$
for some $k \in \{17,19,21,23\}$. If $|E(L^*)|=\frac{w(u-w+1)}{2}$, $\Delta(L^*) \leq w-8$ and
$\deg_{L^*}(x) \equiv u+1 \mod{2}$ for each $x \in V(L^*)$, then there is a triangle decomposition
of $L^* \vee K_w$.
\end{lemma}

\begin{proof} Let $U=V(L^*)$ and let $T=\{z_1,\ldots,z_{w-5}\}$ be a set of vertices which is disjoint
from $U$. By Lemma \ref{singleSparseLeaveHelper}, there is a decomposition
$\{F,F_1,\ldots,F_{w-7}\}$ of $L^*$ such that
\begin{itemize}
    \item
$|E(F)|=u-w+1$ and each component of $F$ is an even-length path or an even-length cycle;
    \item
$F_i$ is a matching for $i \in \{1,\ldots,w-7\}$; and
    \item
there is a proper subset $D$ of $U \setminus V(F)$ such that $|D|=\frac{5w-3u-13}{2}$, $|D
\setminus V(F_{w-7})| \geq 2$, and $\deg_{L^*}(x) \leq w-10$ for each $x \in D$.
\end{itemize}
Let $D=\{d_1,\ldots,d_s\}$, where $s = \frac{5w-3u-13}{2}$, and let
    $$\mathcal{C}_0 =
    \left\{
    \begin{array}{ll}
        \{(d_1,z_{w-5},d_2,z_{w-6})\}, & \hbox{if $|D|=2$;} \\
        \{(d_1,z_{w-7},d_2,z_{w-6},d_3,z_{w-5})\}, & \hbox{if $|D|=3$;} \\
        \{(d_1,z_{w-5},d_2,z_{w-6}),(d_3,z_{w-5},d_4,z_{w-6})\}, & \hbox{if $|D|=4$;} \\
        \{(d_1,z_{w-7},d_2,z_{w-6},d_3,z_{w-5}),(d_4,z_{w-5},d_5,z_{w-6})\}, & \hbox{if $|D|=5$.}
    \end{array}
    \right.$$
Note that $\bigcup_{C \in \mathcal{C}_0}E(C) \subseteq E(K_{D,\{z_{w-6},z_{w-5}\}}) \cup
\{d_1z_{w-7},d_2z_{w-7}\}$ and that each vertex in $D$ is in exactly one cycle in $\mathcal{C}_0$.
Let
    $$\mathcal{T}_0 = \{(x,y,z_i):i \in \{1,\ldots,w-7\}, xy \in E(F_i)\}.$$
It is routine to verify that $\mathcal{T}_0 \cup \mathcal{C}_0$ is a cycle packing of $L^* \vee
K_T$ whose leave is $F \cup G_0$ for some subgraph $G_0$ of $K_{U,T} \cup K_T$.

Choose two distinct vertices $a_1,a_2 \in D$ (note that $|D| \geq 2$). Define a function $\tau: U
\rightarrow \{1,2,3\}$ by
    $$\tau(x)= \left\{
    \begin{array}{ll}
        \deg_{F}(x)+1, & \hbox{if $x \in V(F)$;} \\
        3, & \hbox{if $x \in D \setminus \{a_1,a_2\}$;} \\
        1, & \hbox{otherwise.}
    \end{array}
    \right.$$

Note that we have $\deg_{G_0}(x)=(w-5)-\deg_{L^*-F}(x)$ for each $x \in U \setminus D$ and
$\deg_{G_0}(x)=(w-7)-\deg_{L^*-F}(x)$ for each $x \in D$. Thus, because $\deg_{L^*}(x) \equiv u+1
\mod{2}$ for each $x \in U$ and $w \equiv u+1 \mod {2}$, it can be seen that $\deg_{G_0}(x) \equiv
\tau(x) \mod{2}$ for each $x \in U$.  Further, using the fact that $\deg_{L^*}(x) \leq w-8$ for
each $x \in U \setminus D$ and $\deg_{L^*}(x) \leq w-10$ for each $x \in D$, we have that
$\deg_{G_0}(x) \geq 3 \geq \tau(x)$ for each $x \in U$. Also, using the definition of $\tau$ and
our expressions for $|D|$, $|E(L^*)|$ and $|E(F)|$, we obtain
$$\medop\sum_{x \in U}(\deg_{G_0}(x)-\tau(x))=u(w-5)-2|D|-2(|E(L^*)|-|E(F)|)-\medop\sum_{x \in U}\tau(x)=2\mbinom{w-5}{2}.$$

Let $\mathcal{P}_0=\mathcal{T}_0 \cup \mathcal{C}_0$ and $r=\binom{w-5}{2}$. We claim that for
each $i \in \{0,\ldots,r\}$ there is a packing $\mathcal{P}_i$ of $L \vee K_{T}$ with leave $F
\cup G_i$ for some subgraph $G_i$ of $K_{U,T} \cup K_T$ such that the following hold.
\begin{itemize}
    \item
$|E(G_i \cap K_T)|=\binom{w-5}{2}-i$.
    \item
$\deg_{G_i}(x) \equiv \tau(x) \mod{2}$ and $\deg_{G_i}(x) \geq \tau(x)$ for each $x \in U$.
    \item
Every cycle in $\mathcal{P}_i$ that contains an edge in $K_T$ is a triangle with one vertex in
$V(L)$.
    \item
$\sum_{x \in U} (\deg_{G_i}(x) - \tau(x)) = 2(\binom{w-5}{2}-i)$.
    \item
Each cycle of length at least 4 in $\mathcal{P}_i$ is a subgraph of $K_{V(L),T}$, each vertex in
$D$ is in exactly one such cycle, and each vertex in $V(L) \setminus D$ is not in any such
cycle.
\end{itemize}
To see that these cycle packings exist suppose inductively that a cycle packing $\mathcal{P}_k$
with the required properties exists for some $k \in \{0,\ldots,r-1\}$. We will show that we can
construct a packing $\mathcal{P}_{k+1}$ with the required properties. The properties of
$\mathcal{P}_k$ imply that there is a vertex $y_k \in U$ such that $\deg_{G_i}(y_k) \geq
\tau(y_k)+2$. Let $\mathcal{P}_{k+1}$ be the packing of $L \vee K_{T}$ obtained by applying Lemma
\ref{getTriangle} to $\mathcal{P}_{k}$, taking $y=y_k$, and note that by Lemma \ref{getTriangle}
there is a triangle $C_{k+1} \in \mathcal{P}_{k+1}$ such that $V(C_{k+1}) \cap U = \{y_k\}$ and
$\mathcal{P}_{k+1} \setminus \{C_{k+1}\}$ is equivalent to $\mathcal{P}_{k}$ on $U$. From this it
can be seen that $\mathcal{P}_{k+1}$ has the required properties (note that
$\deg_{G_{k+1}}(y_{k})=\deg_{G_{k}}(y_{k})-2$ and $\deg_{G_{k+1}}(x)=\deg_{G_{k}}(x)$ for each $x
\in U \setminus \{y_{k}\}$).

Consider the cycle packing $\mathcal{P}_r$. From its properties we have that $E(G_r \cap
K_T)=\emptyset$, that $\deg_{G_r}(x) = \tau(x)$ for each $x \in U$, that each cycle of length at
least 4 in $\mathcal{P}_i$ is a subgraph of $K_{D,T}$ and that each vertex in $D$ is in exactly
one such cycle. Thus we can apply Lemma \ref{finishOff} to $\mathcal{P}_r$ to obtain a triangle
decomposition of $L^* \vee K_w$.
\end{proof}



Combining Lemmas \ref{makeLeaveSparse} and \ref{singleSparseLeaveEmbedding}, we obtain the
following result.

\begin{lemma}\label{singleEmbedding}
Let $L$ be a graph of order $u \geq 62$, let $w$ be an integer such that $w = \tfrac{3u+k}{5}$ for
some $k \in \{17,19,21,23\}$, and suppose that $(L,w)$ is admissible. If $|E(L)| \geq
\binom{u}{2}-\frac{w(u-w+1)-u-2}{4}$, there is a triangle decomposition of $L \vee K_w$.

\end{lemma}

\begin{proof}
We can apply Lemma \ref{makeLeaveSparse} to obtain a triangle packing $\mathcal{P}$ of $L$ with a
leave $L^*$ such that $|E(L^*)|=\frac{w(u-w+1)}{2}$ and $\Delta(L^*) \leq w-8$. Because $(L,w)$ is
admissible, we obviously also have $\deg_{L^*}(x) \equiv u+1$ for each $x \in V(L^*)$. So we can
apply Lemma \ref{singleSparseLeaveEmbedding} to obtain a triangle decomposition $\mathcal{D}$ of
$L^* \vee K_w$. Clearly $\mathcal{P} \cup \mathcal{D}$ is a triangle decomposition of $L \vee
K_w$. \end{proof}

\section{Main Result}\label{mainResultSection}

By applying Lemma \ref{singleEmbedding} with $L$ chosen judiciously, we can obtain our main
results without too much further effort.

\begin{proof}[{\rm\bf Proof of Theorem \ref{MainTheorem}}]
Let $(U,\mathcal{A})$ be a partial Steiner triple system of order $u \geq 62$ such that
$|\mathcal{A}| \leq \frac{u^2}{50}-\frac{11u}{100}-\frac{116}{75}$, and let $v$ be an integer such
that $v \geq \frac{8u+17}{5}$ and $v \equiv 1,3 \mod{6}$. Let $L$ be the leave of
$(U,\mathcal{A})$. Let $u'=\frac{5v-k}{8}$ and $w'=\frac{3v+k}{8}$ where $k=21,23,17,19$ when $v
\equiv 1,3,5,7 \mod{8}$ respectively. It is easy to see that $u'$ and $w'$ are integers such that
$u'+w'=v$, $u' \geq u \geq 62$ and $w' = \tfrac{3u'+k}{5}$. Consider the graph $L'=L \vee
K_{u'-u}$ ($L'=L$ if $u'=u$) and note that $(L',w')$ is admissible because $(L,v-u)$ is
admissible. It is routine to check that $|\mathcal{A}| \leq
\frac{u^2}{50}-\frac{11u}{100}-\frac{116}{75}$ implies that $|E(L')| \geq
\binom{u'}{2}-\frac{w'(u'-w'+1)-u'-2}{4}$. Thus we can apply Lemma \ref{singleEmbedding} to
produce a triangle decomposition of $L' \vee K_{w'}$. Because $L' \vee K_{w'}$ is isomorphic to $L
\vee K_{v-u}$, the proof is complete.
\end{proof}

\begin{proof}[\rm\bf Proof of Theorem \ref{NWTheorem}]
Note that $v$ is odd because $G$ is even and contains vertices of degree $v-1$. Let
$u'=\frac{5v-k}{8}$ and $w'=\frac{3v+k}{8}$ where $k=21,23,17,19$ when $v \equiv 1,3,5,7 \mod{8}$
respectively. It is easy to see that $u'$ and $w'$ are integers such that $u'+w'=v$, $u' \geq 62$
and $w' = \tfrac{3u'+k}{5}$. Our hypotheses imply that $G$ has at least $w'$ vertices of degree
$v-1$, so $G$ is isomorphic to $L' \vee K_{w'}$ for some graph $L'$ of order $u'$. It follows from
the facts that $G$ is even and $|E(G)| \equiv 0 \mod 3$ that $(L',w')$ is admissible. It is
routine to check that $|E(G)| \geq \binom{v}{2}-(\frac{3v^2}{128}-\frac{27v}{64}-\frac{409}{128})$
implies that $|E(L')| \geq \binom{u'}{2}-\frac{w'(u'-w'+1)-u'-2}{4}$. So we can apply Lemma
\ref{singleEmbedding} to obtain a triangle decomposition of $L' \vee K_{w'}$. Because $L' \vee
K_{w'}$ is isomorphic to $G$, the proof is complete.
\end{proof}


We conclude by establishing the existence of partial Steiner triple systems with specified numbers
of triples which do not have any embeddings of order close to $u$.

\begin{lemma}\label{noEmbed}
Let $u$ and $w$ be positive integers such that $u+w$ is odd and $w \leq u-5$. There is a partial
Steiner triple system $(U,\mathcal{A})$ of order $u$ such that every embedding of
$(U,\mathcal{A})$ has order at least $u+w$ and
$$|\mathcal{A}|= \left\{
  \begin{array}{ll}
    \tfrac{1}{6}(3u+w^2-4w-3), & \hbox{if $w \equiv 1,3 \mod{6}$;} \\[0.1cm]
    \tfrac{1}{6}(3u+w^2-4w+13), & \hbox{if $w \equiv 5 \mod{6}$;} \\[0.1cm]
    \tfrac{1}{6}(3u+w^2-2w-3), & \hbox{if $w \equiv 0,2 \mod{6}$;} \\[0.1cm]
    \tfrac{1}{6}(3u+w^2-2w+1), & \hbox{if $w \equiv 4 \mod{6}$.}
  \end{array}
\right.$$
\end{lemma}

\begin{proof} Let $U$ be a set of $u$ vertices, let $a$ be a vertex in $U$ and let $S$ be a subset of $U
\setminus \{a\}$ with $|S|=w$. If $w$ is odd then let $S'=S$, and if $w$ is even then let $S'$ be
a subset of $U \setminus \{a\}$ such that $S \subseteq S'$ and $|S'|=w+1$. Let $\mathcal{A}_1$ be
a triangle decomposition of $K_{\{a\}} \vee M$, where $M$ is a matching with vertex set $U
\setminus (S \cup \{a\})$ (note that $|U \setminus (S \cup \{a\})|=u-w-1$ is even because $u+w$ is
odd). Observe that $|S'| \equiv 1,3 \mod{6}$ if $w \equiv 0,1,2,3 \mod{6}$ and $|S'| \equiv 5
\mod{6}$ if $w \equiv 4,5 \mod{6}$. Thus, using the main result of \cite{Sc}, there is a triangle
packing $\mathcal{A}_2$ of $K_{S'}$, such that if $w \equiv 0,1,2,3 \mod{6}$ then the leave of
$\mathcal{A}_2$ is empty, and if $w \equiv 4,5 \mod{6}$ then the leave of $\mathcal{A}_2$ is a
cycle $C=(c_1,c_2,c_3,c_4)$ of length $4$ such that $S' \setminus S = \{c_1\}$ when $w \equiv 4
\mod{6}$. It follows from $w \leq u-5$ that $|U \setminus (S' \cup \{a\})| \geq 3$. Let $z_1$ and
$z_2$ be distinct vertices in $U \setminus (S' \cup \{a\})$ and let
$$\mathcal{A}_3 =
    \left\{
    \begin{array}{ll}
        \emptyset, & \hbox{if $w \equiv 0,1,2,3 \mod{6}$;} \\
        \{(z_1,c_1,c_2),(z_1,c_3,c_4),(z_2,c_2,c_3),(z_2,c_1,c_4))\}, & \hbox{if $w \equiv 5 \mod{6}$;} \\
        \{(z_1,c_3,c_4),(z_2,c_2,c_3))\}, & \hbox{if $w \equiv 4 \mod{6}$.}
    \end{array}
    \right.$$

Let $\mathcal{A}=\mathcal{A}_1 \cup \mathcal{A}_2 \cup \mathcal{A}_3$ and observe that
$(U,\mathcal{A})$ is a partial Steiner triple system with a leave $L$ such that $|\Nbd_L(a)|=w$
and the subgraph of $L$ induced by $\Nbd_L(a)$ is empty. Thus, it can be deduced from Lemma
\ref{NecConds} that every embedding of $(U,\mathcal{A})$ has order at least $u+w$. In each case,
simple counting shows that $\mathcal{A}$ contains the required number of triples.
\end{proof}

\begin{corollary}\label{LBCor}
For any positive integers $u$ and $t$ such that $\frac{u+1}{2} \leq t < \frac{1}{6}(u^2-5u+16)$,
there is a partial Steiner triple system of order $u$ with at most $t$ triples that does not have
an embedding of order $u+\sqrt{6t-3u}-1$ or smaller.
\end{corollary}

\begin{proof}
Let $w$ be the smallest integer such that $w>\sqrt{6t-3u}-1$ and $u+w$ is odd. Note that
$\frac{u+1}{2} \leq t < \frac{1}{6}(u^2-5u+16)$ implies $0 < \sqrt{6t-3u}-1 < u-5$, and hence that
$1 \leq w \leq u-5$. By Lemma \ref{noEmbed} there is a partial Steiner triple system
$(U,\mathcal{A})$ of order $u$ such that every embedding of $(U,\mathcal{A})$ has order at least
$u+w$ and
$$|\mathcal{A}|= \left\{
  \begin{array}{ll}
    \tfrac{1}{6}(3u+w^2-4w-3), & \hbox{if $w \equiv 1,3 \mod{6}$;} \\[0.1cm]
    \tfrac{1}{6}(3u+w^2-4w+13), & \hbox{if $w \equiv 5 \mod{6}$;} \\[0.1cm]
    \tfrac{1}{6}(3u+w^2-2w-3), & \hbox{if $w \equiv 0,2 \mod{6}$;} \\[0.1cm]
    \tfrac{1}{6}(3u+w^2-2w+1), & \hbox{if $w \equiv 4 \mod{6}$.}
  \end{array}
\right.$$ It only remains to show that $|\mathcal{A}| \leq t$.

When $w \neq 5$ it is easy to check that $|\mathcal{A}| \leq \tfrac{1}{6}(3u+w^2-2w+1)$. Thus,
since $w \leq \sqrt{6t-3u}+1$, it can be seen that $|\mathcal{A}| \leq t$. When $w = 5$, we have
$|\mathcal{A}|=\frac{u+6}{2}$. Further, it follows from the definition of $w$ that $\sqrt{6t-3u}-1
\geq 3$, and hence that $t \geq \frac{u+6}{2}$ (note that $t$ is an integer).
\end{proof}

As mentioned in the introduction, it follows from Corollary \ref{LBCor} that for sufficiently
large values of $u$ there is a partial Steiner triple system of order $u$ with at most
$\frac{u^2}{50}-\frac{11u}{100}-\frac{116}{75}$ triples that does not have an embedding of order
$v$ for any $v < (1.346)u$.

Applying Lemma \ref{noEmbed} with $w=2$ demonstrates that, for each admissible integer $u \geq 7$,
there is a partial Steiner triple system of order $u$ with $\frac{u-1}{2}$ triples that has no
completion. This raises the following question which bears some resemblance to Evans' conjecture
for partial latin squares \cite{Eva} (proved in \cite{Sme}).

\begin{question}
Does every partial Steiner triple system of admissible order $u$ with fewer than $\frac{u-1}{2}$
triples have a completion?
\end{question}


\begin{thebibliography}{99}

\bibitem{AnHiMe} L.D. Andersen, A.J.W Hilton and E. Mendelsohn, Embedding partial Steiner triple
    systems, Proc. London Math. Soc. {\bf 41} (1980), 557--576.

\bibitem{Br1} D. Bryant, A conjecture on small embeddings of partial Steiner triple systems, J.
    Combin. Des. {\bf 10} (2002), 313--321.

\bibitem{Br3uEmbed} D. Bryant, Embeddings of partial Steiner triple systems, J. Combin. Theory Ser.
    A {\bf 106} (2004), 77--108.

\bibitem{BrHoEmbed} D. Bryant and D. Horsley, A proof of Lindner's conjecture on embeddings of
    partial Steiner triple systems, J. Combin. Designs {\bf 17} (2009), 63--89.

\bibitem{BrHoMa} D. Bryant, D. Horsley and B. Maenhaut, Decompositions into $2$-regular subgraphs
    and equitable partial cycle decompositions, J. Combin. Theory Ser. B {\bf 93} (2005),
    67--72.

\bibitem{BrHoSubsys} D. Bryant and D. Horsley, Steiner triple systems with two
    disjoint subsystems, J. Combin. Des. {\bf 14} (2006), no. 1, 14--24.

\bibitem{BrMaQuWe} D. Bryant, B. Maenhaut, K. Quinn and B. S. Webb, Existence and embeddings of
    partial Steiner triple systems of order ten with cubic leaves, Discrete Math. {\bf 284} (2004),
    83--95.

\bibitem{Co} C.J. Colbourn, Embedding partial Steiner triple systems is NP-complete,
    J. Combin. Theory Ser. A {\bf 35} (1983), 100--105.

\bibitem{CoCoRo} C.J. Colbourn, M.J. Colbourn and A. Rosa, Completing small partial triple systems,
    Discrete Math. {\bf 45} (1983), 165--179.

\bibitem{CoRo} C.J. Colbourn and A. Rosa, Triple Systems, Clarendon Press, Oxford (1999).

\bibitem{EliFeiSha} P. Elias, A. Feinstein, and C. Shannon, A note on the maximum
    flow through a network, IRE T. Inform. Theor. {\bf 2} (1956), 117-119.

\bibitem{Eva} T. Evans, Embedding incomplete latin squares, Amer. Math. Monthly {\bf 67} (1960),
    958--961.

\bibitem{ForFul} L.R. Ford and D.R. Fulkerson, Maximal flow through a network, Canad. J.
    Math. {\bf 8} (1956), 399–404.

\bibitem{Ha} P. Hall, On Representatives of Subsets, J. London Math. Soc. {\bf 10} (1935), 26--30.


\bibitem{HoEmbed} D. Horsley, Small embeddings of partial Steiner triple systems, J. Combin. Des.
    (to appear).

\bibitem{Ki} T.P. Kirkman, On a problem in combinations, Cambridge and Dublin Math. J. {\bf 2}
    (1847), 191--204.

\bibitem{Ko} D. K{\"o}nig, {\"U}ber Graphen und ihre Anwendung auf Determinantentheorie und
    Mengenlehre, Math. Ann. {\bf 77} (1916), 453--465.

\bibitem{Li} C.C. Lindner, A partial Steiner triple system of order $n$ can be embedded in a
    Steiner triple system of order $6n+3$, J. Combin. Theory Ser. A {\bf 18} (1975), 349--351.

\bibitem{LiEv} C.C. Lindner and T. Evans, Finite embedding theorems for partial designs and
    algebras, SMS 56, Les Presses de l'Universit\'{e} de Montr\'{e}al, (1977).


\bibitem{Nas} C.St.J.A. Nash-Williams, An unsolved problem concerning decomposition of graphs into
    triangles, Combinatorial Theory and its Applications III, Colloquia Math. Soc. J. Bolyai {\bf 4} (1970), 1179--1181.

\bibitem{Sc} J. Sch{\"o}nheim, On maximal systems of $k$-tuples, Studia Sci. Math. Hungar. {\bf 1}
    (1966), 363--368.

\bibitem{Sme} B. Smetaniuk, A new construction on Latin squares. I. A proof of the Evans
    conjecture, Ars Combin. {\bf 11} (1981), 155--172.

\bibitem{Tr} C. Treash, The completion of finite incomplete Steiner triple systems with
    applications to loop theory. J. Combin. Theory Ser. A {\bf 10} (1971), 259--265.


\bibitem{Vi} V.G. Vizing, On an estimate of the chromatic class of a $p$-graph (in Russian),
    Diskret Analiz {\bf 3} (1964), 25--30.


\end{thebibliography}
\end{document}